\documentclass[12pt]{amsart}
\usepackage{amssymb}
\usepackage{verbatim}
\usepackage{enumerate}
\usepackage{amsthm}
\usepackage{url}

\title{Properness under closed forcing}
\date{}
\author{Yasuo Yoshinobu}
\thanks{The author was supported by Grant-in-Aid for Scientific Research (C) (No. 18K03394) from JSPS}
\address{\newline
Graduate School of Information Science\newline
Nagoya University\newline
Furo-cho, Chikusa-ku, Nagoya 464-8601\newline
JAPAN}
\email{yosinobu@i.nagoya-u.ac.jp}

\theoremstyle{definition}
\newtheorem{dfn}{Definition}

\newtheorem{thm}[dfn]{Theorem}
\newtheorem{lma}[dfn]{Lemma}

\newtheorem{clm}{(Claim)}

\newcommand{\force}{\Vdash}

\makeatletter
\renewcommand{\p@enumii}{}
\makeatother
\begin{document}
\maketitle              
\begin{abstract}
For every uncountable regular $\kappa$, we give two examples of proper posets which turn improper in some $\kappa$-closed forcing extension.
\end{abstract}
\section{Introduction}
In this paper we prove the following theorem:
\begin{thm}\label{thm:main}
Let $\kappa$ be an uncountable regular cardinal. Then there exists a $\kappa$-closed poset $\mathbb{Q}$ and a proper poset $\mathbb{P}$ such that
$$
  \force_{\mathbb{Q}}\text{\lq\lq$\check{\mathbb{P}}$ is not proper.\rq\rq}
$$
\end{thm}

This theorem negatively answers the question raised by Karagila \cite{karagila_web}, which asks if there exists a sufficiently large cardinal $\kappa$ such that any proper poset remains proper after any $\kappa$-closed forcing. The motivation of this question is discussed in a recent article by Asper\'{o} and Karagila \cite{aspero-karagila}.

In fact, we give two examples witnessing the conclusion of Theorem \ref{thm:main}. The first example is a rather simple mimic of the known example of a pair of proper posets whose product is improper, given by Shelah \cite[XVII Observation 2.12, p.826]{shelah98:_proper_improper}. This example is briefly mentioned in \cite{aspero-karagila}. The second example involves the class of posets introduced by Moore \cite{moore05:_set} in connection with the Mapping Reflection Principle ($\mathrm{MRP}$) he introduced in the same paper. While the second example requires a longer argument, it has the advantage that $\mathbb{P}$ is taken to be totally proper. 
\section{The First Example}
Let $T$ be the complete binary tree of height $\kappa$. Let $\mathbb{Q}=T$ with the reversed ordering. Clearly $\mathbb{Q}$ is $\kappa$-closed. Let $\mathbb{P}_0=\mathrm{Add}(\omega)$ and $\dot{\mathbb{P}}_1=\mathrm{Col}(\omega_1, 2^\kappa)$ in $V^{\mathbb{P}_0}$. $\dot{\mathbb{P}}_1$ is $\sigma$-closed in $V^{\mathbb{P}_0}$. Note that, no new cofinal branches are added to $T$ by forcing with $\mathbb{P}_0*\dot{\mathbb{P}}_1$, by a well-known argument first proposed by Mitchell \cite{Aronszajn-transfer}. Let $\dot{C}$ be a $(\mathbb{P}_0*\dot{\mathbb{P}}_1)$-name for a cofinal subset of $\kappa$ of order type $\omega_1$, and let $T\upharpoonright\dot{C}$ denote the subset of $T$ consisting of all nodes of level in $\dot{C}$ (defined in $V^{\mathbb{P}_0*\dot{\mathbb{P}}_1}$). Note that, in $V^{\mathbb{P}_0*\dot{\mathbb{P}}_1}$, $T\upharpoonright\dot{C}$ forms a tree of size and height $\omega_1$, and since every cofinal branch through $T\upharpoonright\dot{C}$ generates one through $T$, which is in $V$, the number of cofinal branches through $T\upharpoonright\dot{C}$ is also $\omega_1$. Now let $\dot{\mathbb{P}}_2$ be a $(\mathbb{P}_0*\dot{\mathbb{P}}_1)$-name for the c.c.c. poset specializing $T\upharpoonright\dot{C}$, as described in Baumgartner \cite[\S 7]{baumgartner84:_applic_proper_forcin_axiom}, and set $\mathbb{P}:=\mathbb{P}_0*\dot{\mathbb{P}}_1*\dot{\mathbb{P}}_2$. $\mathbb{P}$ is a three step iteration of c.c.c., $\sigma$-closed and c.c.c. posets, and thus is proper. Note that, since $T\upharpoonright\dot{C}$ is specialized in $V^\mathbb{P}$, $\omega_1$ must be collapsed in any further extension where new cofinal branches through $T\upharpoonright\dot{C}$ are added. Since forcing with $\mathbb{Q}$ over $V$ adds a new cofinal branch through $T$, forcing with $\mathbb{Q}\times\mathbb{P}$ adds a cofinal branch through $T\upharpoonright\dot{C}$ which is not in $V^\mathbb{P}$. Therefore $\mathbb{Q}\times\mathbb{P}$ collapses $\omega_1$, and thus is improper. This shows that $\mathbb{P}$ is improper in $V^\mathbb{Q}$.
\section{The Second Example}
First let us review some relevant definitions and facts we will use in construction of our second example. Let us start with the notion of total properness.
\begin{dfn}
Let $\mathbb{P}$ be a poset, and $N$ a countable $\in$-model of a suitable fragment of $\mathrm{ZFC}$ which contains $\mathbb{P}$.
\begin{enumerate}[(1)]
  \item $p\in\mathbb{P}$ is said to be {\it $(N, \mathbb{P})$-generic\/} if for every dense subset $D\in N$ of $\mathbb{P}$, $D\cap N$ is predense below $p$.
  \item $p\in\mathbb{P}$ is said to be {\it totally $(N, \mathbb{P})$-generic\/} if for every dense subset $D\in N$ of $\mathbb{P}$, $p$ extends some element of $D\cap N$.
\end{enumerate}
\end{dfn}

Note that $\mathbb{P}$ is proper iff for every sufficiently large regular $\theta$, every countable $N\prec H(\theta)$ containing $\mathbb{P}$ and every $p\in N\cap\mathbb{P}$, there exists an $(N, \mathbb{P})$-generic condition of $\mathbb{P}$ which extends $p$.

\begin{dfn}
A poset $\mathbb{P}$ is said to be {\it totally proper\/} if for every sufficiently large regular $\theta$, every countable $N\prec H(\theta)$ containing $\mathbb{P}$ and every $p\in N\cap\mathbb{P}$, there exists a totally $(N, \mathbb{P})$-generic condition of $\mathbb{P}$ which extends $p$.
\end{dfn}

The notion of total properness has the following simple characterization.

\begin{thm}[Eisworth-Roitman\cite{eisworth_roitman99:_ostaszewski}]\label{thm:ER}
A poset $\mathbb{P}$ is totally proper iff $\mathbb{P}$ is proper and $\sigma$-Baire\footnote{$\sigma$-Baire posets are sometimes referred to as $\sigma$-distributive posets.}.\qed
\end{thm}

Theorem \ref{thm:ER} can be proved using the following lemma, which we will use later.

\begin{lma}\label{lma:totalextend}
Suppose $\mathbb{P}$ is $\sigma$-Baire and $N$ is countable. Then every $(N, \mathbb{P})$-generic condition can be extended to a totally $(N, \mathbb{P})$-generic condition.\qed
\end{lma}

Next let us review the class of posets first introduced by Moore \cite{moore05:_set}. They were originally designed to prove the Mapping Reflection Principle ($\mathrm{MRP}$) from the Proper Forcing Axiom ($\mathrm{PFA}$). In our second example, $\mathbb{P}$ will be taken from this class.

\begin{dfn}[Moore \cite{moore05:_set}]
Let $\theta$ be a regular uncountable cardinal, and $X\in H(\theta)$. For a club subset $\mathcal{E}$ of $\{M\in[H(\theta)]^\omega\mid M\prec H(\theta)\}$, a function $\Sigma:\mathcal{E}\to\mathcal{P}([X]^\omega)$ is said to be an {\it open stationary set mapping\/} if for every $M\in\mathcal{E}$,
\begin{enumerate}[(1)]
  \item $\Sigma(M)$ is $M$-stationary, that is, $\Sigma(M)\cap C\cap M\not=\emptyset$ for every club subset $C\in M$ of $[X]^\omega$, and
  \item $\Sigma(M)$ is open in the Ellentuck topology, that is, for every $x\in\Sigma(M)$ there exists a finite $a\subseteq x$ such that
           $$
             [a, x]:=\{y\in[x]^\omega\mid a\subseteq y\}\subseteq\Sigma(M).
           $$
\end{enumerate}

For an open stationary set mapping $\Sigma:\mathcal{E}\to\mathcal{P}([X]^\omega)$, the poset $\mathcal{P}_{\Sigma}$ is defined as follows: $\mathcal{P}_{\Sigma}$ consists of the functions of the form $q:\alpha+1\to\mathcal{E}$ for some $\alpha<\omega_1$ such that
\begin{enumerate}[(1)]
  \item $q(\gamma)\in q(\gamma+1)$ for every $\gamma<\alpha$,
  \item $q(\gamma)=\bigcup_{\xi<\gamma}q(\xi)$ for every limit $\gamma\leq\alpha$, and
  \item for every limit $\gamma\leq\alpha$, there exists $\nu<\gamma$ such that $q(\xi)\cap X\in\Sigma(q(\gamma))$ for all $\xi$ satisfying $\nu<\xi<\gamma$.
\end{enumerate}
$\mathcal{P}_{\Sigma}$ is ordered by initial segment.
\end{dfn}

The following lemma is the heart of Moore's proof of $\mathrm{MRP}$ from $\mathrm{PFA}$.

\begin{lma}[Moore \cite{moore05:_set}]\label{lma:totallyproper}
$\mathcal{P}_{\Sigma}$ is totally proper.\qed
\end{lma}

The following are density lemmata about $\mathcal{P}_{\Sigma}$, which we will use later. For $q\in\mathcal{P}_{\Sigma}$, we write $q(\mathrm{dom}(q)-1)$ as $\mathrm{last}(q)$.

\begin{lma}[Moore \cite{moore05:_set}]\label{lma:density}
Let $\Sigma:\mathcal{E}\to\mathcal{P}([X]^\omega)$ be an open stationary set mapping, where $\mathcal{E}$ is a club subset of $\{M\in[H(\theta)]^\omega\mid M\prec H(\theta)\}$ for an uncountable regular cardinal $\theta$. 
\begin{enumerate}[(1)]
  \item For every $\alpha<\omega_1$, $D_\alpha:=\{q\in\mathcal{P}_{\Sigma}\mid\alpha\in\mathrm{dom}(q)\}$ is dense in $\mathcal{P}_{\Sigma}$.\label{item:dalpha}
  \item For every $x\in H(\theta)$, $E_x:=\{q\in\mathcal{P}_{\Sigma}\mid x\in \mathrm{last}(q)\}$ is dense in $\mathcal{P}_{\Sigma}$.\label{item:ex}\qed
\end{enumerate}
\end{lma}

We also use the following lemma in our construction.

\begin{lma}\label{lma1}
Let $\kappa$ be an uncountable regular cardinal.
For every club subset $C$ of $[\kappa]^\omega$, there exists a club subset $\tilde{C}$ of $C$ such that
\begin{enumerate}[(a)]
  \item $\tilde{C}$ is closed in the Ellentuck topology.\label{cond1}
  \item $C\setminus \tilde{C}$ is weakly unbounded, that is, for every $a\in[\kappa]^{<\omega}$ there exists $x\in C\setminus \tilde{C}$ such that $a\subseteq x$.\label{cond2}
\end{enumerate}
\end{lma}

\begin{proof}
Let $f:{}^{<\omega}\kappa\to\kappa$ be such that
$$
  C(f)=\{x\in[\kappa]^\omega\mid \text{$x$ is closed under $f$}\}\subseteq C.
$$
Let
$$
  \tilde{C}:=C(f)\setminus\{\mathrm{cl}_f(a)\mid a\in[\kappa]^{<\omega}\},
$$
where $\mathrm{cl}_f(a)$ denotes the closure of $a$ by $f$. Since every union of a strictly $\subseteq$-increasing $\omega$-sequence of elements of $C(f)$ is in $\tilde{C}$, it easily follows that $\tilde{C}$ is a club subset of $C$. If $x\in[\kappa]^\omega$ is an accumulation point of $\tilde{C}$ in Ellentuck topology, $x$ must be closed under $f$ and also cannot be finitely generated by $f$, and thus is in $\tilde{C}$. This shows (\ref{cond1}). Since $a\subseteq\mathrm{cl}_f(a)\in C\setminus \tilde{C}$ holds for every $a\in[\kappa]^{<\omega}$, we have (\ref{cond2}).\qed
\end{proof}

Now we start to describe our construction of the second example. Suppose a regular uncountable cardinal $\kappa$ is given.
Let $\mathcal{C}$ be the set of club subsets of $[\kappa]^\omega$, and let $\mathbb{Q}=Col(\kappa, |\mathcal{C}|)$. Clearly $\mathbb{Q}$ is $\kappa$-closed, and $|\mathcal{C}|=\kappa$ holds in $V^\mathbb{Q}$. Let $\dot{C}$ be the $\mathbb{Q}$-name for a diagonal intersection of the members of $\mathcal{C}$. $\dot{C}$ is a club subset of $[\kappa]^\omega$ in $V^\mathbb{Q}$.

Let $\theta$ be a sufficiently large regular cardinal, and let
$$
  \mathcal{E}=\{M\in[H(\theta)]^\omega\mid\mathbb{Q}, \dot{C}\in M\prec H(\theta)\}.
$$
Fix an arbitrary $M\in\mathcal{E}$. Pick a totally $(M, \mathbb{Q})$-generic condition $p_M\in\mathbb{Q}$. Let $\{C^M_n\mid n<\omega\}$ enumerate $M\cap\mathcal{C}$. Note that, for each $n<\omega$, we may assume that $\tilde{C}^M_n\in M\cap\mathcal{C}$, where $\tilde{C}^M_n$ is made from $C^M_n$ as in Lemma \ref{lma1}. By the definition of $\dot{C}$, for each $n<\omega$ there exists a $\mathbb{Q}$-name $\dot{\alpha}^M_n\in M$ (for an ordinal below $\kappa$) such that
\begin{equation}\label{eqn1}
\force_{\mathbb{Q}}\text{\lq\lq $\{x\in\dot{C}\mid\dot{\alpha}^M_n\in x\}\subseteq(\tilde{C}^M_n)\check{}$.\rq\rq}
\end{equation}
By (\ref{cond1}) and (\ref{cond2}) of Lemma \ref{lma1}, there exists a $\mathbb{Q}$-name $\dot{y}^M_n$ and $\dot{a}^M_n$ both in $M$ such that
\begin{equation}\label{eqn2}
\force_{\mathbb{Q}}\text{\lq\lq $\dot{\alpha}^M_n\in\dot{a}^M_n\in[\dot{y}^M_n]^{<\omega}\land\dot{y}^M_n\in (C^M_n\setminus\tilde{C}^M_n)\check{}\land [\dot{a}^M_n, \dot{y}^M_n]\cap(\tilde{C}^M_n)\check{}=\emptyset$,\rq\rq}
\end{equation}
%where (for $y\in[\kappa]^\omega$ and $a\in[y]^{<\omega}$) $[a, y]$ denotes the Ellentuck open interval $\{x\in[y]^\omega\mid a\subseteq x\}$.
By (\ref{eqn1}) and (\ref{eqn2}) we have
$$
  \force_{\mathbb{Q}}\text{\lq\lq $[\dot{a}^M_n, \dot{y}^M_n]\cap\dot{C}=\emptyset$.\rq\rq}
$$
By the total genericity of $p_M$, for each $n<\omega$ there exists $y^M_n\in C^M_n\cap M$ and $a^M_n\in[y^M_n]^{<\omega}$ such that
\begin{equation}\label{eqn4}
p_M\force_{\mathbb{Q}}\text{\lq\lq $[(a^M_n)\check{}, (y^M_n)\check{}\ ]\cap\dot{C}=\emptyset$.\rq\rq}
\end{equation}
Now let us set $\Sigma(M):=\bigcup_{n<\omega}[a^M_n, y^M_n]$. It is easy to check that $\Sigma$ is an open stationary set mapping ($\Sigma(M)$ is clearly open, and is $M$-stationary since $y^M_n\in\Sigma(M)\cap C^M_n\cap M$ for every $n<\omega$). By (\ref{eqn4}) we have
\begin{equation}\label{eqn5}
p_M\force_{\mathbb{Q}}\text{\lq\lq $(\Sigma(M))\check{}\cap\dot{C}=\emptyset$.\rq\rq}
\end{equation}
Now let $\mathbb{P}:=\mathcal{P}_\Sigma$. $\mathbb{P}$ is totally proper by Lemma \ref{lma:totallyproper}.

\begin{lma}\label{lma2}
There exists $p\in\mathbb{Q}$ such that
$$
  p\force_{\mathbb{Q}}\text{\lq\lq $\{M\in\check{\mathcal{E}}\mid p_M\in\dot{G}\}$ is stationary in $[({H(\theta)}^V)\check{}\ ]^\omega$, \rq\rq}
$$
where $\dot{G}$ is the canonical $\mathbb{Q}$-name for $\mathbb{Q}$-generic filter.
\end{lma}

\begin{proof}
Suppose otherwise. Then there exists a $\mathbb{Q}$-name $\dot{D}$ such that
$$
  \force_{\mathbb{Q}}\text{\lq\lq $\dot{D}$ is a club subset of $\check{\mathcal{E}}$ and $\forall M\in\dot{D}[p_M\notin\dot{G}]$.\rq\rq}
$$
Let $\lambda$ be a sufficiently large regular cardinal, and let $N$ be a countable elementary submodel of $H(\lambda)$ such that $\dot{D}\in N$ and $M:=N\cap H(\theta)\in\mathcal{E}$. Now let $G$ be a $\mathbb{Q}$-generic filter over $V$ such that $p_M\in G$. Then by genericity of $p_M$ it holds that $N[G]\cap (H(\theta))^V=N\cap (H(\theta))^V=M$. On the other hand, since $\dot{D}\in N$ and $\dot{D}$ is forced to be a club, it holds that $N[G]\cap (H(\theta))^V\in\dot{D}_G$. Thus it follows that $M\in\dot{D}_G$, and $p_M\notin G$ should be the case. This is a contradiction.\qed
\end{proof}

Now let $p$ be as in Lemma \ref{lma2}. We will show that
\begin{equation}\label{eqn:improper}
  p\force_{\mathbb{Q}}\text{\lq\lq$\check{\mathbb{P}}$ is not proper.\rq\rq}
\end{equation}
Note that it is enough for our goal: we may rename $\mathbb{Q}$ below $p$ as $\mathbb{Q}$, or use the weak homogeneity of $\mathbb{Q}$.

To this end, let $G$ be a $\mathbb{Q}$-generic filter over $V$ such that $p\in G$. Work in $V[G]$. Let
$$
\mathcal{F}:=\{M\in\mathcal{E}\mid M\cap\kappa\in\dot{C}_G\}.
$$
Since $\dot{C}_G$ is a club subset of $[\kappa]^\omega$, $\mathcal{F}$ is a club subset of $\mathcal{E}\subseteq[(H(\theta))^V]^\omega$.

Now for each $\alpha<\omega_1$, let
$$
\mathcal{D}_\alpha:=\{q\in\mathbb{P}\mid\exists\xi\in\mathrm{dom}(q)[\xi\geq\alpha\land q(\xi)\in\mathcal{F}]\}.
$$
It is clear that $\mathcal{D}_\alpha$ is a dense subset of $\mathbb{P}$ (by Lemma \ref{lma:density}(\ref{item:dalpha}), and the unboundedness of $\mathcal{F}$; note that any $q\in\mathbb{P}$ can be extended by adding any member of $\mathcal{E}$ which contains $\mathrm{last}(q)$ as an element).

Now let $S:=\{M\in\mathcal{E}\mid p_M\in G\}$. By our choice of $p$, $S$ is a stationary subset of $[(H(\theta))^V]^\omega$. 

Now let $\lambda$ be a sufficiently large regular caridnal, and set
$$
\mathcal{S}:=\{N\in[H(\lambda)]^\omega\mid\mathbb{P}, (\mathcal{D}_\alpha)_{\alpha<\omega_1}\in N\prec H(\lambda)\land N\cap(H(\theta))^V\in S\}.
$$
Then $\mathcal{S}$ is stationary in $[H(\lambda)]^\omega$. Pick any $N\in\mathcal{S}$. Set $M:=N\cap(H(\theta))^V\in S$ and $\delta:=N\cap\omega_1=M\cap\omega_1$. The following claim is enough to show (\ref{eqn:improper}).

\begin{clm}
There are no $(N, \mathbb{P})$-generic conditions.
\end{clm}

\begin{proof}
Suppose there is an $(N, \mathbb{P})$-generic condition. Since $\mathbb{P}$ remains $\sigma$-Baire in $V[G]$, by Lemma \ref{lma:totalextend} such a condition can be extended to a totally $(N, \mathbb{P})$-generic condition $q$. By the total genericity of $q$, Lemma \ref{lma:density}(\ref{item:dalpha})(\ref{item:ex}), the facts that the domain of every condition of $\mathbb{P}\cap N$ is less than $\delta$, and that $\mathcal{D}_\alpha\in N$ for every $\alpha<\delta$, we may assume that
\begin{enumerate}[\rm(a)]
  \item $\mathrm{dom}(q)=\delta+1$,\label{item:domq}
  \item $\mathrm{last}(q)=q(\delta)=M$ and\label{item:lastq}
  \item $\{\xi<\delta\mid q(\xi)\in\mathcal{F}\}$ is unbounded in $\delta$.\label{item:unbounded}
\end{enumerate}
By (\ref{item:unbounded}) we have
$$
  \text{$\{\xi<\delta\mid q(\xi)\cap\kappa\in\dot{C}_G\}$ is unbounded in $\delta$.}
$$
On the other hand, by (\ref{item:lastq}) there exists $\nu<\delta$ such that
$$
  \text{$q(\xi)\cap\kappa\in\Sigma(M)$ holds for all $\xi$ satisfying $\nu<\xi<\delta$.}
$$
But since $M\in S$, $p_M\in G$ holds, and thus by (\ref{eqn5}) we have $\Sigma(M)\cap\dot{C}_G=\emptyset$. This is a contradiction.\qed
\end{proof}

This finishes our construction.

Note that our second example is more delicate than the first one in the following sense: although the properness of $\mathbb{P}$ is destroyed by forcing over $\mathbb{Q}$, the product $\mathbb{Q}\times\mathbb{P}$ remains proper unlike the first example, because $\mathbb{P}$ is $\sigma$-Baire and therefore $\mathbb{Q}$ remains $\sigma$-closed by forcing over $\mathbb{P}$.

\bibliographystyle{plain}
\bibliography{locco}
\end{document}